\documentclass[12pt]{article}

\usepackage{amsmath,amsthm,amsfonts,amscd,amssymb}
\usepackage[pdfborder={0 0 0}]{hyperref}
\usepackage{enumerate,verbatim}
\usepackage{enumitem}
\usepackage{mathtools}
\usepackage{amsthm}

\usepackage{pdfpages}
\usepackage{graphicx}
\usepackage{epsf}
\usepackage{amsmath,amssymb,latexsym }
\usepackage{ latexsym, amsfonts, delarray, amsmath, delarray}
\usepackage{bbm}
\usepackage{tikz}
\usetikzlibrary{positioning}
\usetikzlibrary{arrows}
\usetikzlibrary{decorations.pathreplacing}
\usetikzlibrary{decorations.markings}

\newcommand{\ZZ}{\mathbb{Z}}

\newtheorem{thm}{Theorem}[section]
\newtheorem{cor}[thm]{Corollary}
\newtheorem{lem}[thm]{Lemma}
\newtheorem{prop}[thm]{Proposition}

\theoremstyle{definition}

\theoremstyle{remark}

\theoremstyle{remark}

\DeclareMathOperator{\Z}{\mathbb{Z}}
\DeclareMathOperator{\R}{\mathbb{R}}
\DeclareMathOperator{\F}{\mathbb{F}}
\DeclareMathOperator{\Q}{\mathbb{Q}}
\DeclareMathOperator{\C}{\mathbb{C}}
\DeclareMathOperator{\V}{\mathcal{V}}

\DeclareMathOperator{\Aut}{Aut}

\DeclareMathOperator{\Hom}{Hom}
\DeclareMathOperator{\GL}{GL}
\DeclareMathOperator{\Img}{Im}

\usepackage{fullpage}
\begin{document}
\begin{center}
{\Large \bf  Spherical 2-designs and lattices from Abelian groups}

\vspace{7mm}
{\large Albrecht B\"ottcher, Simon Eisenbarth,\\[0.5ex]

Lenny Fukshansky, Stephan Ramon Garcia, Hiren Maharaj}
\end{center}

\bigskip
\begin{quote}
{\bf Abstract.}
We consider lattices generated by finite Abelian groups. We prove that such a lattice is strongly eutactic,
which means the normalized minimal vectors of the lattice form a spherical 2-design, if and only if the group is of
odd order or if it is a power of the group of order 2. This result also yields a criterion for
the appropriately normalized minimal vectors to constitute a uniform normalized tight frame. Further, our result combined with a
recent theorem of R. Bacher produces (via the classical Voronoi criterion) a new infinite family of extreme lattices. Additionally,
we investigate the structure of the automorphism groups of these lattices, strengthening our previous results in this
direction.

\let\thefootnote\relax\footnote{\hspace*{-7.5mm} MSC 2010: Primary 11H06, 05B30, 52C17; Secondary 11H31, 11H50, 42C15}
\let\thefootnote\relax\footnote{\hspace*{-7.5mm} Keywords: strongly eutactic lattice, spherical 2-design, uniform normalized tight frame}
\let\thefootnote\relax\footnote{\hspace*{-7.5mm}  Fukshansky acknowledges support by NSA grant H98230-1510051, Garcia acknowledges support by NSF grant DMS-1265973.}

\end{quote}

\def\A{{\mathcal A}}
\def\AA{{\mathfrak A}}
\def\B{{\mathcal B}}
\def\C{{\mathbb C}}
\def\D{{\mathcal D}}
\def\EE{{\mathfrak E}}
\def\F{{\mathcal F}}
\def\G{{\mathcal G}}
\def\x{{\mathcal H}}
\def\I{{\mathcal I}}
\def\II{{\mathfrak I}}
\def\J{{\mathcal J}}
\def\K{{\mathcal K}}
\def\kk{{\mathfrak K}}
\def\L{{\mathcal L}}
\def\LL{{\mathfrak L}}
\def\M{{\mathcal M}}
\def\mm{{\mathfrak m}}
\def\MM{{\mathfrak M}}
\def\N{{\mathcal N}}
\def\O{{\mathcal O}}
\def\OO{{\mathfrak O}}
\def \P{{\mathfrak P}}
\def \R {{\mathbb R}}
\def\W{{\mathcal W}}
\def \PNR{{\mathcal P_N(\real)}}
\def\PMNR{{\mathcal P^M_N(\real)}}
\def\PdNR{{\mathcal P^d_N(\real)}}
\def\S{{\mathcal S}}
\def\V{{\mathcal V}}
\def\X{{\mathcal X}}
\def\Y{{\mathcal Y}}
\def\Z{{\mathbb Z}}
\def\H{{\mathcal H}}
\def\cee{{\mathbb C}}
\def\Nn{{\mathbb N}}
\def\pee{{\mathbb P}}
\def\que{{\mathbb Q}}
\def\QQ{{\mathbb Q}}
\def\real{{\mathbb R}}
\def\bS{{\mathbb S}}
\def\RR{{\mathbb R}}
\def\zed{{\mathbb Z}}
\def\ZZ{{\mathbb Z}}
\def\aaa{{\mathbb A}}
\def\ff{{\mathbb F}}
\def\HDelta{{\it \Delta}}
\def\kk{{\mathfrak K}}
\def\qbar{{\overline{\mathbb Q}}}
\def\kbar{{\overline{K}}}
\def\ybar{{\overline{Y}}}
\def\kkbar{{\overline{\mathfrak K}}}
\def\ubar{{\overline{U}}}
\def\eps{{\varepsilon}}
\def\ahat{{\hat \alpha}}
\def\bhat{{\hat \beta}}
\def\k{{\nu}}
\def\gt{{\tilde \gamma}}
\def\h{{\tfrac12}}
\def\be{{\boldsymbol e}}
\def\bff{{\boldsymbol f}}
\def\bei{{\boldsymbol \e_i^\T }}
\def\bc{{\boldsymbol c}}
\def\bm{{\boldsymbol m}}
\def\bk{{\boldsymbol k}}
\def\bi{{\boldsymbol i}}
\def\bl{{\boldsymbol l}}
\def\bq{{\boldsymbol q}}
\def\bu{{\boldsymbol u}}
\def\bt{{\boldsymbol t}}
\def\bs{{\boldsymbol s}}
\def\bv{{\boldsymbol v}}
\def\bw{{\boldsymbol w}}
\def\bx{{\boldsymbol x}}
\def\bbx{{\overline{\boldsymbol x}}}
\def\bX{{\boldsymbol X}}
\def\bz{{\boldsymbol z}}
\def\bwy{{\boldsymbol y}}
\def\bY{{\boldsymbol Y}}
\def\bL{{\boldsymbol L}}
\def\ba{{\boldsymbol a}}
\def\bb{{\boldsymbol b}}
\def\bet{{\boldsymbol\eta}}
\def\bxi{{\boldsymbol\xi}}
\def\bo{{\boldsymbol 0}}
\def\bone{{\boldsymbol 1}}
\def\bol{{\boldsymbol 1}_L}
\def\ep{\varepsilon}
\def\p{\boldsymbol\varphi}
\def\q{\boldsymbol\psi}
\def\rank{\operatorname{rank}}
\def\aut{\operatorname{Aut}}
\def\lcm{\operatorname{lcm}}
\def\sgn{\operatorname{sgn}}
\def\spn{\operatorname{span}}
\def\md{\operatorname{mod}}
\def\Norm{\operatorname{Norm}}
\def\dim{\operatorname{dim}}
\def\det{\operatorname{det}}
\def\Vol{\operatorname{Vol}}
\def\rk{\operatorname{rk}}
\def\ord{\operatorname{ord}}
\def\ker{\operatorname{ker}}
\def\div{\operatorname{div}}
\def\Gal{\operatorname{Gal}}
\def\GL{\operatorname{GL}}
\def \SNR{\operatorname{SNR}}
\def\WR{\operatorname{WR}}
\def\IWR{\operatorname{IWR}}
\def\scg{\operatorname{\left< \Gamma \right>}}
\def\swrh{\operatorname{Sim_{WR}(\Lambda_h)}}
\def\ch{\operatorname{C_h}}
\def\cht{\operatorname{C_h(\theta)}}
\def\scgt{\operatorname{\left< \Gamma_{\theta} \right>}}
\def\scgmn{\operatorname{\left< \Gamma_{m,n} \right>}}
\def\gat{\operatorname{\Omega_{\theta}}}
\def\Obar{\operatorname{\overline{\Omega}}}
\def\Lbar{\operatorname{\overline{\Lambda}}}
\def\mn{\operatorname{mn}}
\def\disc{\operatorname{disc}}
\def\rot{\operatorname{rot}}
\def\Prob{\operatorname{Prob}}
\def\co{\operatorname{co}}
\def\ot{\operatorname{o_{\tau}}}
\def\Aut{\operatorname{Aut}}
\def\Mat{\operatorname{Mat}}
\def\SL{\operatorname{SL}}
\def\id{\operatorname{id}}

\def\calP{\mathcal{P}}
\def\u{{\bf u}}
\def\e{{\bf e}}
\def\E{{\bf \overline{e}  }}

\def\T {{\mathrm{T}}}

\section{Introduction and main result}
\label{introd}

A collection of points $\bwy_1,\dots,\bwy_m$ on the unit sphere $\bS_{n-1}$ in $\real^n$ is called a {\it spherical $t$-design} for some integer $t \geq 1$ if for every polynomial $f(\bX) = f(X_1,\dots,X_n)$ with real coefficients of degree $\leq t$ the equality
\begin{equation}
\int_{\bS_{n-1}} f(\bX)\ d \mu(\bX) = \frac{1}{m} \sum_{i=1}^m f(\bwy_i)\label{deft}
\end{equation}
holds, where $\mu$ is the surface measure normalized so that $\mu(\bS_{n-1})=1$. Spherical designs  were introduced in the
  celebrated 1977 paper~\cite{delsarte} of Delsarte, Goethals, and Seidel and  have been studied extensively ever since for their remarkable properties and many applications within and outside of mathematics.  The strong connection between spherical designs and lattices  was first observed by B. B. Venkov. See in particular~\cite{venkov} (also surveyed in~\cite{martinet}, Chapter~16)
 and the nice survey of Venkov's fundamental work on this subject written by Nebe~\cite{nebe}.

Let $\Lambda \subset \real^n$ be a lattice of full rank. The {\it minimal norm} of $\Lambda$ is defined as
$$|\Lambda| = \min \left\{ \|\bx\| : \bx \in \Lambda \setminus \{ \bo \} \right\},$$
where $\|\ \|$ denotes the Euclidean norm, and we let
$$S(\Lambda) = \left\{ \bx \in \Lambda : \|\bx\| = |\Lambda| \right\}$$
denote the {\it set of minimal vectors} of $\Lambda$. Clearly $S(\Lambda)$ is a symmetric set, and the set
$$S'(\Lambda) := \frac{1}{|\Lambda|} S(\Lambda) = \left\{ \frac{1}{|\Lambda|} \bx: \bx \in \Lambda , \|\bx\| = |\Lambda| \right\}$$
is a finite subset of the unit sphere $\bS_{n-1}$. The lattice $\Lambda$ is called {\it strongly eutactic} if $S'(\Lambda)$ is a spherical 2-design.

We remark that the original definition of strongly eutactic lattices is different. However, it is equivalent to the one given here, that is, to the property that $S'(\Lambda)$ is a spherical 2-design. We refer to~\cite{martinet} (Section~3.2 and Chapter~16, especially Corollary~16.1.3) and~\cite{nebe} for further information on this. There are several equivalent criteria for spherical designs, and hence for the strong eutaxy condition on lattices, e.g. Venkov's criterion (Proposition~16.1.2 and Theorem~16.1.4 of~\cite{martinet}) as well as Theorem~4.1 of \cite{lenny}.

We will use the convenient criterion of Theorem 16.1.2 of \cite{martinet}, which states that a lattice $\Lambda$ in dimension $n$ is strongly eutactic if and only if
\begin{equation}
\sum_{\bx \in S(\Lambda)} (\bx,\bwy)^2 =   \frac{|\Lambda|^2 |S(\Lambda)|}{n}(\bwy,\bwy)\label{mar}
\end{equation}
for all $\bwy \in \spn_{\real} \Lambda$, where $(\ ,\ )$ stands for the usual dot product of vectors and  $|S(\Lambda)|$ is the cardinality of the set $S(\Lambda)$.
In fact, it
is sufficient to check the associated bilinear form for a basis. Namely, if
$\{\bb_1, \ldots, \bb_n\}$ is a basis of $\spn_{\real} \Lambda$, then $\Lambda$ is strongly eutactic if and
only if
\begin{equation}
\sum_{\bx \in S(\Lambda)} (\bx,\bb_i)(\bx,\bb_j)=\frac{|\Lambda|^2 |S(\Lambda)|}{n}(\bb_i,\bb_j)\label{marie}
\end{equation}
for all $i,j \in \{1,\ldots,n\}$.

Strongly eutactic lattices are important in lattice theory due to their central role in discrete optimization problems, especially sphere packing. For instance, A. Sch\"urmann recently proved~\cite{achill} that all perfect strongly eutactic lattices are periodic extreme, i.e., these lattices cannot be ``locally modified'' to yield a better periodic packing (recall that a full-rank lattice $\Lambda$ in $\real^n$ is called perfect if the set of $n \times n$ symmetric matrices $\{ \bx \bx^\top : \bx \in S(\Lambda) \}$ spans the space of all $n \times n$ symmetric matrices as a real vector space). In fact, it has been proved by Voronoi around 1900 that a lattice is extreme (i.e., is a local maximum of  the packing density function on the space of lattices in a fixed dimension) if and only if it is perfect and {\it eutactic},  a condition being weaker than strong eutaxy; we refer the reader to Martinet's book~\cite{martinet} and Nebe's paper~\cite{nebe} for definitions and further information. While many of the standard lattices, such as indecomposable root lattices, are known to be strongly eutactic, a full classification of strongly eutactic lattices is only known in small dimensions. This makes constructions of strongly eutactic lattices in arbitrary dimensions particularly interesting.

In this paper, we revisit the family of lattices generated by finite Abelian groups that we studied previously in~\cite{bo_et_al}, showing that many of them are strongly eutactic. Here is our main result.

\begin{thm} \label{MAIN}
Let $G=\{g_1 := 0,g_2,\ldots,g_{n}\}$ be a finite {\rm(}additively written{\rm)} Abelian group of order $n \ge 2$,
let $L_G$ be the sublattice of the root lattice
$$A_{n-1}=\left\{\bx  =(x_1,\ldots,x_{n}) \in \mathbb{Z}^n: \sum_{i=1}^n x_i =0\right\}$$
which is defined by
\begin{equation}
L_G=\left\{\bx  =(x_1,\ldots,x_{n}) \in A_{n-1}: \sum_{j=2}^{n} x_j g_j=0\right\} \label{alb1}.
\end{equation}
Then $L_G$ is strongly eutactic in $\spn_\real L_G$ if and only if $G$ has odd order or $G$ is isomorphic to $(\zed/2\zed)^\nu$ for some $\nu \geq 1$.
\end{thm}

The Abelian group lattices $L_G$ were also recently studied by R. Bacher~\cite{bacher}, who proved that when $|G| \geq 9$, the lattice $L_G$ is perfect. Combining Bacher's Theorem~5.3, Example~5.2.1 and Example~5.2.2 of~\cite{bacher} with our Theorem~\ref{MAIN} and using Voronoi's criterion, we obtain the following corollary. 

\begin{cor} \label{extreme} Suppose that $|G| \geq 7$ is odd or $G \cong (\zed/2\zed)^{\nu}, \nu \geq 3$. Then the lattice $L_G$ is extreme.
\end{cor}

\noindent
In fact, based on some computational evidence we conjecture that all lattices $L_G$ are eutactic, and hence extreme (notice that we prove that when $|G|$ is even and $G$ is not isomorphic to a power of $\zed/2\zed$ then $L_G$ is not strongly eutactic, but it still appears likely to be eutactic).

We also want to mention the connection between spherical $2$-designs and frames.
Let $m \geq n$  and let $\bx_1,\dots,\bx_m \in \real^n$ be a collection of vectors of norm $\sqrt{n/m}$ such that
\[\real^n = \spn_{\real} \{ \bx_1,\dots,\bx_m \} \;\:\mbox{and}\;\:
\sum_{i=1}^m (\bx_i, \bwy)^2 = \|\bwy\|^2\;\:\mbox{for each}\;\: \bwy \in \real^n.\]
Such a collection of vectors is called a {\em uniform normalized tight {\rm(}UNT{\rm)} $(m,n)$-frame}. It is well-known (see, for instance, Proposition 1.2 of~\cite{paulsen}) that a finite subset $\bwy_1,\dots,\bwy_m$ of the unit sphere $\bS_{n-1}$ in $\real^n$ is a spherical 2-design if and only if $\sqrt{n/m}\ \bwy_1,\dots, \sqrt{n/m}\ \bwy_m$ is a UNT
$(m,n)$-frame and $\sum_{i=1}^m \bwy_i = \bo$. In fact, this observation was actually made earlier by B. B. Venkov~\cite{venkov}: it is a special case of a more general criterion, which works for all $t$ (see Proposition 16.1.2 of~\cite{martinet}). Since sets of minimal vectors of lattices are $\bo$-symmetric, the condition $\sum_{i=1}^m \bwy_i = \bo$ is automatically satisfied, and so our lattices also provide a family of UNT $(m,n)$-frames, where $m$ is the number of minimal vectors in the corresponding lattice, as given by~\eqref{numminvecs} below.

In addition to the eutaxy properties, we revisit the automorphism groups of the lattices~$L_G$. Let us recall that the automorphism group of the lattice $L_G$ consists of all the orthogonal transformations of the space $\spn_\real L_G$ that permute $L_G$, hence it is a finite group. We previously proved in~\cite{bo_et_al} that the automorphism group of the lattice $L_G$ contains as a subgroup the automorphism group of the group $G$, more precisely
$\Aut(L_G) \cap S_{n-1} \cong \Aut(G),$
where $n = |G|$ and $S_{n-1}$ is the symmetric group.  Here we strengthen our earlier result.

\begin{thm} \label{auto} Let $G$ be an Abelian group of order $|G| = n \ge 3$ and let $C_k$ denote the cyclic group of order $k$. Then
$$\Aut(L_G) \cong C_2 \times ( G \rtimes \Aut(G) )$$
if and only if $n \in \{3,4,6\}$ or $n \ge 12$.
\end{thm}

We remark that $C_2 \times ( G \rtimes \Aut(G) )$ is always a subgroup of $\Aut(L_G)$. The cases $3 \le n \le 14$
are disposed of by the following table, which was calculated with Magma \cite{magma}.

\begin{table}[!htb]
    \begin{minipage}{.5\linewidth}
		\begin{tabular}{c | c}
			$G \cong$                    & $\tfrac{|\Aut(L_G)|}{| C_2 \times ( G \rtimes \Aut(L_G) ) |}$  \\
			\hline
			$C_3$                       	& $1$ \\
			$C_4$                        & $1$ \\
			$C_2 \times C_2$ 			& $1$ \\
			$C_5$                         & $6$ \\
			$C_6$                         & $1$ \\
			$C_7$                         & $8$ \\
			$C_8$                         & $4$ \\
			$C_2 \times C_4$              & $24$ \\
			$C_2 \times C_2 \times C_2$ 	& $240$
		\end{tabular}
    \end{minipage}%
    \begin{minipage}{.5\linewidth}
		\begin{tabular}{c | c}
			$G \cong$                     & $\tfrac{|\Aut(L_G)|}{| C_2 \times ( G \rtimes \Aut(L_G) ) |}$  \\
			\hline
			$C_9$                         & $6$ \\
			$C_3 \times C_3$              & $72$ \\
			$C_{10}$                      & $6$ \\
			$C_{11}$                      & $12$ \\
			$C_{12}$                      & $1$ \\
			$C_{2} \times C_{6}$     	 & $1$ \\
			$C_{13}$                      & $1$ \\
			$C_{14}$                      & $1$
		\end{tabular}
    \end{minipage}
\end{table}

In Section~\ref{main_statement},  we recall the formula for the number of minimal vectors in the lattice~$L_G$  and
give the proof of Theorem~\ref{MAIN}, split into several parts. We discuss automorphisms of $L_G$ and prove Theorem~\ref{auto} for $n \ge 15$ in Section~\ref{automorphisms}.

\section{Strong eutaxy}
\label{main_statement}

For convenience, when the context is clear,  we will refer to the lattice point $(x_1,\ldots,x_{n})$
by using the formal sum $x_1 g_1 + x_2 g_2 + \ldots + x_n g_n$ in the
group ring $\mathbb{Z}[G]$. The following result is proved in \cite{fu_ma} in the
special case of lattices from elliptic curves. Subsequently, in \cite{bo_et_al} we pointed out
that this  theorem is valid for the lattices $L_G$ with virtually no change to the proof.
We include this argument here for the purposes of self-containment.

\begin{prop} \label{minum}
Assume that $n\ge 4$ and let $\kappa$ denote the
order of the subgroup $G_2 := \{x \in G: 2x = 0 \}$ of $G$.
Then the number of minimal vectors in $L_G$ is
\begin{equation}\label{numminvecs}
\frac{n}{\kappa} \cdot \frac{(n-\kappa)(n-\kappa-2)}{4} +
\left (n-\frac{n}{\kappa} \right)\cdot \frac{n(n-2)}{4}.
\end{equation}
\end{prop}

\begin{proof}
As shown in \cite{bo_et_al, fu_ma}, every minimal vector of $L_G$ is of the form $p+q-r-s$ where $p,q,r,s \in G$ are distinct
and $p+q = r+s$. Consider the homomorphism $\tau:G \to G$ defined by $\tau(p) = 2p$. The kernel of $\tau$ is the
subgroup $G_2$ of $G$ and the image $\mathrm{Im}(\tau)$ of $\tau$  has $n/\kappa$ points.
Fix an element  $z$ of $G$. First we count the number of solutions to the equation
$p + q = z$ where $p,q $ are distinct elements of $G$.  Observe that
$p = q$ if and only if $z \in \mathrm{Im}(\tau)$.

If $z \in \mathrm{Im}(\tau) $, there are $\kappa$ solutions $p$ to $2p = z$.
Thus there are $n-\kappa$ possible  $p$ such that
$q := z-p \ne p$, and so there are $(n-\kappa)/2$ pairs $p,q$ such that $p + q = z$
and $p \ne q$. Hence the number of pairs $r,s$ disjoint from $\{p,q\}$ and such that $r + s = z$ is $(n-\kappa -2)/2$.
In total, there are $(n-\kappa)/2 \cdot (n-\kappa -2)/2 = (n-\kappa)(n-\kappa-2)/4$
possible minimal vectors $p+q-r-s$ such that $p+q = z = r + s$.
The size of the image of $\tau$ is $\frac{n}{\kappa}$  so the total number of
possible minimal vectors $p+q-r-s$ such that $p+q = z = r + s$ with $z \in \mathrm{Im}(\tau)$ is
$\frac{n}{\kappa} \cdot \frac{(n-\kappa)(n-\kappa-2)}{4}.$

If $z  \not \in \mathrm{Im}(\tau) $, there are no solutions $p$ to $2p = z$.
A similar reasoning as above shows that there are
$(n-\frac{n}{\kappa})\cdot \frac{n(n-2)}{4}$  minimal vectors $p+q-r-s$ with
$p + q   \not \in \mathrm{Im}(\tau) $. Thus, by the above argument, the number of minimal vectors of $L_G$ is given by (\ref{numminvecs}).
\end{proof}

We remark that $\kappa=1$ if and only if and only if $n$ is an odd number while $\kappa=n$ if and only if
$G \cong (\mathbb{Z}/2\mathbb{Z})^\nu$ for some $\nu \ge 2$. We now embark on our main result, piece by piece. The cases $|G|=2$ and $|G|=3$ are trivial. For $|G|=2$, the lattice $L_G$ is the root lattice $A_1$ stretched by the factor $2$, and if $|G|=3$, then $L_G$ is the hexagonal lattice $A_2$ stretched by the factor $\sqrt{3}$. Both lattices are known to be strongly eutactic, hence we are left with the case of $|G| \ge 4$.

\begin{thm} \label{step1} Let $|G| = n \ge 4$ be even and suppose $G$ is not isomorphic to a power of $\zed/2\zed$. Then $L_G$ is not strongly eutactic.
\end{thm}

\proof
Working towards a contradiction, assume $L_G$ is strongly eutactic.
Let $\kappa$ be as in Proposition~\ref{minum}. Our hypothesis is equivalent to the restriction $1 < \kappa < n$. The cardinality $|S(L_G)|$ is given by~(\ref{numminvecs}),
and an elementary computation shows that~(\ref{numminvecs})
is equal to
\[\frac{n(n^2-4n+2+\kappa)}{4}=n\frac{(n-1)(n-3)+\kappa-1}{4}.\]
In \cite{bo_et_al}, we proved that the minimal norm $|L_G|$ equals $2$. Pick a minimal vector
$\bwy \in S(L_G)$. As the rank of $L_G$ is $n-1$, we deduce from (\ref{mar}) that
\begin{eqnarray*}
& & \sum_{\bx \in S(L_G)} (\bx,\bwy)^2 =\frac{|L_G|^2|S(L_G)|}{n-1}(\bwy,\bwy)\\
& & =2^2n \frac{(n-1)(n-3)+\kappa-1}{4(n-1)}\cdot 4
=4n(n-3)+\frac{4n(\kappa-1)}{n-1}.
\end{eqnarray*}
The squared inner products $(\bx,\bwy)^2$ are integers and hence $4n(\kappa-1)/(n-1)$ must also be an integer. However, since
$n-1$ is odd, $4n$ and $n-1$ do not have common prime divisors, which means the sum is an integer if and only if $n-1$ divides $\kappa-1$.
But this contradicts the restriction $1 < \kappa < n$.
\endproof

\begin{thm} \label{step2} Let $G \cong ( \mathbb{Z} / p \mathbb{Z})^\nu$ be elementary Abelian. Then $L_G$ is strongly eutactic.
\end{thm}

\proof
The automorphism group $\Aut(L_G)$ contains the holomorph $G \rtimes \Aut(G)$, which acts doubly-transitive on $G$ since $G$ is elementary Abelian. By Theorem~11.6~d) of \cite{huppert}, the associated permutation character decomposes into two irreducible characters, so $\mathbb{R}^{p^\nu}$ decomposes into two subspaces on which $G \rtimes \Aut(G)$ acts irreducibly. One of them is $\spn_{\real} \{ (1,\dots,1) \}$, so the other one is its complement $\spn_\real L_G$. The assertion now follows from Theorem~3.6.6 in~\cite{martinet}.
\endproof

\begin{thm} \label{step3} Let $|G| = n \ge 5$ be odd. Then $L_G$ is strongly eutactic.
\end{thm}

Our proof requires two lemmas and goes as follows. For $z \in G$ let $t_z \in \lbrace 1,\dots,n \rbrace$ be the unique element with $2g_{t_z} = z$ and put
\begin{align*}
R_z & := \lbrace (a,b) \in \lbrace 1,\dots,n \rbrace^2 : g_a + g_b = z, a \neq b \rbrace \\
S_z & := \lbrace (a,b) \in \lbrace 1,\dots,n \rbrace^2 : g_a + g_b = z \rbrace = R_z \cup \lbrace (t_z,t_z) \rbrace.
\end{align*}
We have $S_z = \lbrace (a,\varphi(a)) : a \in \lbrace 1,\dots,n \rbrace \rbrace$ for a bijection $\varphi$ whose unique fixed point is $t_z$.

The vectors $\bff_1 := \be_1 - \be_2, \dots, \bff_{n-1} := \be_1 - \be_n$  form a basis for $\spn_{\real} A_{n-1} = \spn_{\real} L_G$. Let $A \in \mathbb{R}^{n \times (n-1)}$ be the matrix whose $i$-th column is the (transposed) vector $\bff_i$, so
\begin{displaymath}
A = \left(\begin{array}{rrrrr}
1  &  1 & \dots & 1  &  1 \\
-1 &  0 & \dots & 0  &  0 \\
0  & -1 & \dots & 0  &  0 \\
   &    & \ddots &    &    \\
0  &  0 & \dots & -1 &  0 \\
0  &  0 & \dots &  0 & -1
\end{array}\right).
\end{displaymath}

\noindent
The following identities can be established by direct verification.

\begin{lem} \label{ident}
For every $j,k \in \lbrace 1,\dots,n-1 \rbrace$ the following identities hold:
\begin{enumerate}
\setlength{\itemsep}{8pt}
\item $ \sum\limits_{ a=1 }^n{ A_{a,j} } = 0 $.
\item $ \sum\limits_{ (a,b) \in S_z } A_{a,j}A_{a,k} = \sum\limits_{ a=1 }^n{ A_{a,j} A_{a,k} } = (\bff_j,\bff_k)$ for all $z \in G$.
\item $ \sum\limits_{ z \in G }{ \sum\limits_{ (a,b) \in S_z }{ A_{a,j}A_{b,k} } } = \sum\limits_{a,b = 1}^n{ A_{a,j}A_{b,k} } = \Big( \sum\limits_{a=1}^n{ A_{a,j} } \Big) \Big( \sum\limits_{b=1}^n{A_{b,k}} \Big) = 0$.
\item $ \sum\limits_{(a,b) \in S_z}{ \sum\limits_{(c,d) \in S_z}{ A_{a,j}A_{c,k} } } = \sum\limits_{(a,b) \in S_z}{A_{a,j} \Big( \sum\limits_{ (c,d) \in S_z }{ A_{c,k} } \Big) } = 0 $ for all $z \in G$.
\item $ \sum\limits_{z \in G} \sum\limits_{(a,b) \in S_z} A_{t_z,j}A_{t_z,k} = n \sum\limits_{z \in G} A_{t_z,j}A_{t_z,k} = n\ (\bff_j,\bff_k)$.
\end{enumerate}
In part 5 the statement is also true for the summands $A_{b,j}A_{c,k}$, $A_{b,j}A_{d,k}$ and $A_{a,j}A_{d,k}$. \\
\end{lem}

If $n$ is odd, then Proposition~\ref{minum} with $\kappa =1$ shows that the number of minimal vectors in $L_G$ is $m = \tfrac{n(n-1)(n-3)}{4}$. Let $S(L_G) = \lbrace \bu_1,\dots,\bu_m \rbrace$ be the set of minimal vectors. Since $n \ge 5$, we know from~\cite{bo_et_al} that the minimal norm in $L_G$ is $2$. Thus, by virtue of~\eqref{marie}, we have to show that
\begin{equation}
\sum\limits_{i=1}^m{ ( \bu_i, \bff_j ) ( \bu_i, \bff_k ) } = 4 \frac{m}{n-1} (\bff_j,\bff_k) = n(n-3) (\bff_j,\bff_k) \label{prev}
\end{equation}
for every $j,k \in \lbrace 1,\dots,n-1 \rbrace$. For $u = \be_a + \be_b - \be_c - \be_d \in S(L_G)$, we have
\begin{displaymath}
(\bu,\bff_j) = A_{a,j} + A_{b,j} - A_{c,j} - A_{d,j}.
\end{displaymath}
Let
\begin{displaymath}
f_{j,k}(a,b,c,d) := ( A_{a,j} + A_{b,j} - A_{c,j} - A_{d,j} )( A_{a,k} + A_{b,k} - A_{c,k} - A_{d,k} ).
\end{displaymath}
Then we get
\begin{align*}
\sum\limits_{i=1}^m{ ( \bu_i , \bff_j )( \bu_i , \bff_k ) } & = \sum\limits_{ z \in G }{ \tfrac{1}{2} \sum\limits_{ (a,b) \in R_z }{ \tfrac{1}{2} \sum\limits_{ \substack{ (c,d) \in \\ R_z \setminus \lbrace (a,b),(b,a) \rbrace } }{ f_{j,k}(a,b,c,d) } } } \\
& = \tfrac{1}{4} \sum\limits_{ z \in G }{ \sum\limits_{ (a,b) \in R_z }{ \Big( \sum\limits_{ (c,d) \in R_z }{ f_{j,k}(a,b,c,d) } - \underbrace{ f_{j,k}(a,b,a,b) }_{=0} - \underbrace{ f_{j,k}(a,b,b,a) }_{=0} \Big) } }. \\
\end{align*}
Consequently, equality \eqref{prev} will follow as soon as we have proved the following lemma.

\begin{lem} \label{sum_lemma}
\begin{displaymath}
\sum\limits_{ z \in G }{ \sum\limits_{ (a,b) \in R_z }{ \sum\limits_{ (c,d) \in R_z }{ f_{j,k}(a,b,c,d) } } } = 4n(n-3) (\bff_j, \bff_k).
\end{displaymath}
\end{lem}

\proof We have
\begin{align*}
& \sum\limits_{ z \in G } \sum\limits_{ (a,b) \in R_z } \sum\limits_{ (c,d) \in R_z }{ f_{j,k}(a,b,c,d) } \\
= & \sum\limits_{ z \in G } \sum\limits_{ (a,b) \in R_z } \Big( \sum\limits_{ (c,d) \in S_z }{ f_{j,k}(a,b,c,d) } - f_{j,k}(a,b,t_z,t_z) \Big)  \\
= & \sum\limits_{ z \in G } \Big( \sum\limits_{ (a,b) \in S_z } \Big[ \sum\limits_{ (c,d) \in S_z }{ f_{j,k}(a,b,c,d) } - f_{j,k}(a,b,t_z,t_z) \Big] \\
& \qquad - \sum\limits_{ (c,d) \in S_z }{ f_{j,k}(t_z,t_z,c,d) + \underbrace{ f_{j,k}(t_z,t_z,t_z,t_z) }_{=0} } \Big) \\
= & \sum\limits_{ z \in G } \sum\limits_{ (a,b) \in S_z  } \sum\limits_{ (c,d) \in S_z }{ f_{j,k}(a,b,c,d) } - \sum\limits_{ z \in G } \sum\limits_{ (a,b) \in S_z }{ f_{j,k}(a,b,t_z,t_z) } - \sum\limits_{ z \in G } \sum\limits_{ (c,d) \in S_z }{ \underbrace{ f_{j,k}(t_z,t_z,c,d) }_{ = f_{j,k}(c,d,t_z,t_z) } } \\
= & \sum\limits_{ z \in G } \sum\limits_{ (a,b) \in S_z } \sum\limits_{ (c,d) \in S_z }{ f_{j,k}(a,b,c,d) } - 2 \sum\limits_{ z \in G }
\sum\limits_{ (a,b) \in S_z }{ f_{j,k}(a,b,t_z,t_z) }.
\end{align*}
We now consider the two sums obtained individually. We expand $f_{j,k}$ and apply Lemma~\ref{ident} to every summand.

\par\noindent
\begin{align*}
\text{(i)} ~ ~ & \sum\limits_{ z \in G } \sum\limits_{ (a,b) \in S_z } \sum\limits_{ (c,d) \in S_z } { f_{j,k}(a,b,c,d) }  \\
= & \sum\limits_{ z \in G  } \sum\limits_{ (a,b) \in S_z } \sum\limits_{ (c,d) \in S_z } \Big( A_{a,j}A_{a,k} + A_{a,j}A_{b,k}  - A_{a,j}A_{c,k} - A_{a,j}A_{d,k} + A_{b,j}A_{a,k}  + A_{b,j}A_{b,k} \\
& \qquad \qquad \qquad \qquad - A_{b,j}A_{c,k} - A_{b,j}A_{d,k} - A_{c,j}A_{a,k} - A_{c,j}A_{b,k} + A_{c,j}A_{c,k} + A_{c,j}A_{d,k} \\
& \qquad \qquad \qquad \qquad - A_{d,j}A_{a,k} - A_{d,j}A_{b,k} + A_{d,j}A_{c,k} + A_{d,j}A_{d,k} \Big)  \\
= & \quad \sum\limits_{ z \in G } \sum\limits_{ (a,b) \in S_z } { n \Big( A_{a,j}A_{a,k} + A_{a,j}A_{b,k} + A_{b,j}A_{a,k} + A_{b,j}A_{b,k} \Big) } \\
& + \sum\limits_{ z \in G } \sum\limits_{ (c,d) \in S_z } { n \Big( A_{c,j}A_{c,k} + A_{c,j}A_{d,k} + A_{d,j}A_{c,k} + A_{d,j}A_{d,k} \Big) } \\
& - \sum\limits_{ z \in G } \sum\limits_{ (a,b) \in S_z } \sum\limits_{ (c,d) \in S_z } \Big( A_{a,j}A_{c,k} + A_{a,j}A_{d,k} + A_{b,j}A_{c,k} + A_{b,j}A_{d,k} + A_{c,j}A_{a,k} + A_{c,j}A_{b,k} \\
& \qquad \qquad \qquad \qquad \quad + A_{d,j}A_{a,k} + A_{d,j}A_{b,k} \Big) \\
= & \sum\limits_{ z \in G }{ 2 n (\bff_j,\bff_k) } + \sum\limits_{ z \in G }{ 2 n (\bff_j,\bff_k) }& \\
= & ~ 4 n^2 (\bff_j,\bff_k)&
\end{align*} \noindent
\begin{flalign*}
\text{(ii)} ~ ~ & \sum\limits_{ z \in G } \sum\limits_{ (a,b) \in S_z }{ f_{j,k}(a,b,t_z,t_z) } & \\
= & \sum\limits_{ z \in G } \sum\limits_{ (a,b) \in S_z }{ \Big( ( A_{a,j} + A_{b,j} - 2 A_{t_z,j} )( A_{a,k} + A_{b,k} - 2 A_{t_z,k} ) \Big) }  \\
= & \sum\limits_{ z \in G } \sum\limits_{ (a,b) \in S_z } \Big( A_{a,j}A_{a,k} + A_{a,j}A_{b,k} - 2A_{a,j}A_{t_z,k} + A_{b,j}A_{a,k} + A_{b,j}A_{b,k} - 2A_{b,j}A_{t_z,k} \\
& \qquad \qquad \quad - 2A_{t_z,j}A_{a,k} - 2A_{t_z,j}A_{b,k} + 4A_{t_z,j}A_{t_z,k} \Big) \\
= & ~ n (\bff_j,\bff_k) + n (\bff_j,\bff_k) + 4 n (\bff_j,\bff_k) \\
= & ~ 6 n (\bff_j,\bff_k)
\end{flalign*}
Thus, our sum is $(4n^2-2\cdot 6n)(\bff_j,\bff_k)=4n(n-3)(\bff_j,\bff_k)$, as asserted.
\endproof

As stated above, the previous lemma completes the proof of Theorem~\ref{step3}.
Putting Theorems~\ref{step1}, \ref{step2}, and~\ref{step3} together, we obtain our main result, Theorem~\ref{MAIN}.

\section{Automorphisms}
\label{automorphisms}

In this section we discuss properties of the automorphism groups of the lattices $L_G$, proving Theorem~\ref{auto}. We start with a preliminary result.

\begin{prop} \label{stab} The stabilizer of $L_G$ in $\Aut(A_{n-1})$ is isomorphic to $C_2 \times ( G \rtimes \Aut(G) )$.
\end{prop}

\proof
Notice that $\Aut(A_{n-1}) \cong C_2 \times S_n$, where $S_n$ acts on the unit vectors $\be_1,\dots,\be_n$ by permutation. Let $\varphi \in S_n$ be an automorphism of $A_{n-1}$ which stabilizes $L_G$. In particular, this is an automorphism of $L_G$. Now, there exists a $\varphi' \in G \leq \Aut(L_G)$ such that $(\varphi' \circ \varphi)(\be_1) = \be_1$. So the composition is an automorphism of $L_G$, permuting the vectors $\be_2,\dots,\be_n$, therefore it is contained in $\Aut(G)$.
\endproof

In particular, we get $C_2 \times ( G \rtimes \Aut(G) ) \leq \Aut(L_G)$ and we want to show that this is the full automorphism group for $|G| \geq 15$. It is sufficient to show that every automorphism of $L_G$ stabilizes $A_{n-1}$, because then every automorphism of $L_G$ is an automorphism of $A_{n-1}$ which stabilizes $L_G$, and therefore it is contained in $C_2 \times ( G \rtimes \Aut(G) )$. However, to show this it is in turn enough to prove that the dual lattices $L_G^{\#}$ and $A_{n-1}^{\#}$ have the same minimal vectors, i.e. $S(L_G^{\#}) = S(A_{n-1}^{\#})$. Indeed, let $\varphi$ be an automorphism of a lattice $L$. For all $\bv \in L^{\#}$ we have
\begin{displaymath}
( \varphi(\bv), L ) = ( \varphi(\bv), \varphi(L) ) = (\bv,L) \subset \Z,
\end{displaymath}
which implies $\varphi(\bv) \in L^{\#}$ and $\varphi$ is an automorphism of $L^{\#}$. Now let $\varphi \in \Aut(L_G)$ be an automorphism of $L_G$. Then $\varphi$ is an automorphism of $L_G^{\#}$ permuting its minimal vectors. By assumption we have $S(L_G^{\#}) = S(A_{n-1}^{\#})$ and since $A_{n-1}^{\#}$ is generated by its minimal vectors, $\varphi$ is an automorphism of $A_{n-1}^{\#}$, and therefore one of $A_{n-1}$ stabilizing $L_G$. This implies that $\Aut(L_G)$ is contained in $C_2 \times ( G \rtimes \Aut(G) )$.

In Section~5.3 of~\cite{martinet}, the automorphism groups of the Barnes lattices are determined. On the other hand, Barnes lattices are precisely our lattices $L_G$ for cyclic groups $G$. The following comes from Theorem~5.3.7 of~\cite{martinet} (note that in the book $n$ is the dimension of the lattice, not the order of the group).

\begin{thm} \label{mincyclic} Let $G$ be cyclic and of order $n \geq 15$. Then we have $S(L_G^{\#}) = S(A_{n-1}^{\#})$.
\end{thm}

Now we will show that this theorem already implies the general case.

\begin{thm} Let $G = \lbrace g_1:=0, g_2, \dots,g_n \rbrace$ be an Abelian group of order $n \geq 15$. Then we have $S(L_G^{\#}) = S(A_{n-1}^{\#})$.
\end{thm}

\proof
First, we show that the residue classes of $L_G^{\#} / A_{n-1}^{\#}$ can be calculated with the dual group $G^* = \Hom(G, \Q / \Z)$ of $G$. Using this construction, we can permute the entries of a class to obtain one of $L_{C_n}^{\#} / A_{n-1}^{\#}$. Then the assertion can be deduced from Theorem~\ref{mincyclic}. Define
\begin{displaymath}
\widetilde{ L_G } = \left\{ (x_1,\dots,x_n) \in \Z^n : \sum_{i=2}^n x_i g_i = 0 \right\}.
\end{displaymath}
This is a lattice in $\Z^n$ with $\Z^n / \widetilde{L_G} \cong G$. The dual group of $G$ is $G^* = \Hom(G,\que / \zed)$ and the dual lattice $\widetilde{L_G}^{\#}$ is generated by $\Z^n$ and the vectors $(\varphi(g_1),\dots,\varphi(g_n))$, $\varphi \in G^*$ (here we can take an arbitrary rational representative of $\varphi(g_i) \in \Q / \Z$). Indeed, for all $(x_1,\dots,x_n) \in \widetilde{L_G}$ and $\varphi \in G^*$ we have
\begin{displaymath}
\sum_{i=1}^n x_i \varphi(g_i) = \sum_{i=1}^n \varphi( x_i g_i ) = \varphi \left( \sum_{i=1}^n x_i g_i \right) = \varphi(0) \in \Z
\end{displaymath}
and
\begin{displaymath}
\left< ~ \Z^n, ~ (\varphi(g_1),\dots,\varphi(g_n)) : \varphi \in G^* ~ \right> / \Z^n \cong G^*,
\end{displaymath}
which means we have found a lattice contained in $\widetilde{ L_G }^{\#}$ with the same determinant.
Now let
\begin{displaymath}
\pi : \real^n \to \spn_\real L_G, \quad x \mapsto \frac{(x,\mathbbm{1})}{n} \cdot \mathbbm{1},
\end{displaymath}
be the projection onto $\spn_\real L_G$, where $\mathbbm{1} \in \real^n$ denotes the all-one vector. All in all, we have
\begin{figure}[h]
\centering
\begin{tikzpicture}
\fill[black] (10,0) circle (0.05);
\fill[black] (10,1) circle (0.05);
\fill[black] (10,2) circle (0.05);
\fill[black] (10,3) circle (0.05);
\fill[black] (0,1) circle (0.05);
\fill[black] (0,2) circle (0.05);
\fill[black] (0,3) circle (0.05);
\draw (10,0) -- (10,1);
\draw (10,1) -- (10,2);
\draw (10,2) -- (10,3);
\draw (0,1) -- (0,2);
\draw (0,2) -- (0,3);

\coordinate[label=right: {$L_G^{\#} = \pi ( \widetilde{L_G}^{\#} ) $}] (LGd) at (10,3);
\coordinate[label=right:{ $A_{n-1}^{\#} = \pi (\Z^n)$ } ] (And) at (10,2);
\coordinate[label=right:{$A_{n-1} = \Z^n \cap \spn_\real L_G$}] (An) at (10,1);
\coordinate[label=right:{$L_G = \widetilde{L_G} \cap \spn_\real L_G$}] (LG) at (10,0);

\coordinate[label=left: \footnotesize $G$] (G) at (10,0.5);
\coordinate[label=left: \footnotesize $C_n$] (G) at (10,1.5);
\coordinate[label=left: \footnotesize {$G^*$}] (G) at (10,2.5);

\coordinate[label=right:{$\widetilde{L_G}^{\#} = \langle ~ \Z^n, ~ ( \varphi(g_1),\dots,\varphi(g_n) ) ~ : ~ \varphi \in G^* ~ \rangle$}] (LGdt) at (0,3);
\coordinate[label=right:{$\Z^n$}] (Zn) at (0,2);
\coordinate[label={[black]right:$\widetilde{L_G} = \lbrace (x_1,\dots,x_n) \in \Z^n ~ : ~ \sum x_i g_i = 0 \rbrace $}] (LGt) at (0,1);

\coordinate[label=left: \footnotesize $G$] (G) at (0,1.5);
\coordinate[label=left: \footnotesize {$G^*$}] (G) at (0,2.5);

\end{tikzpicture}
\end{figure}

We get $L_G^{\#} / A_{n-1}^{\#} = \lbrace ~ \pi (\varphi(g_1),\dots,\varphi(g_n))  + A_{n-1}^{\#} ~ : ~ \varphi \in G^* \rbrace$, and $S(L_G^{\#})$ equals $S(A_{n-1}^{\#})$ if and only if
\begin{displaymath}
\min \left\{\| v \| : v \in \pi (\varphi(g_1),\dots,\varphi(g_n)) + A_{n-1}^{\#} \right\} > \min( A_{n-1}^{\#}) = \sqrt{ \frac{n-1}{n} }
\end{displaymath}
for every $0 \neq \varphi \in G^*$. Consider such a $\varphi$. The group $\Q / \Z$ is isomorphic to the union of all roots of unity in $\C$ (via $\exp( \tfrac{2 \pi i k}{m} ) \mapsto \tfrac{k}{m} + \mathbb{Z}$), so the image of $\varphi$ is isomorphic to a finite subgroup of $\C^*$, and hence is cyclic. This means that the vector
\begin{displaymath}
( \varphi( g_1 ), \dots, \varphi( g_n ) )
\end{displaymath}
contains several copies of $\Img(\varphi)$, the image of $\varphi$. Now let $\Img(\varphi) = \langle q \rangle$ and let $C_n = \langle g \rangle$ be a cyclic group of order $n$. The homomorphism
\begin{displaymath}
\psi : C_n \to \Q / \Z, \quad g^i \mapsto i \cdot q
\end{displaymath}
has the same image as $\varphi$, hence there exists a permutation matrix $M$ with
\begin{displaymath}
( \varphi(g_1), \dots, \varphi(g_n) ) \cdot M = ( \psi(g^0), \psi(g), \dots, \psi(g^{n-1}) ).
\end{displaymath}
This permutation $M$ is an automorphism of $A_{n-1}^{\#}$ and commutes with the projection $\pi$, so we have
\begin{displaymath}
[ ~ \pi (\varphi(g_1), \dots, \varphi(g_n) ) + A_{n-1}^{\#} ~ ] \cdot M = \pi ( \psi(g^0), \psi(g), \dots, \psi( g^{n-1} )) + A_{n-1}^{\#}.
\end{displaymath}
The right-hand side is a residue class of $L_{C_n}^{\#} / A_{n-1}^{\#}$, and from Theorem~\ref{mincyclic} we know that its minimal norm is greater than $\min( A_{n-1}^{\#} )$. This completes the proof.
\endproof

\medskip
{\bf Acknowledgement:} We wish to thank Professor Gabriele Nebe for her helpful comments on the subject of this paper.

\bigskip
A. B\"ottcher, Fakult\"at f\"ur Mathematik, TU Chemnitz, 09107 Chemnitz, Germany

{\tt aboettch@mathematik.tu-chemnitz.de}

\medskip

S. Eisenbarth, Lehrstuhl D f\"ur Mathematik, RWTH Aachen, 52062 Aachen, Germany

{\tt simon.eisenbarth@rwth-aachen.de}

\medskip
L. Fukshansky, Department of Mathematics,  Claremont McKenna College,

850 Columbia Ave,
Claremont, CA 91711, USA

{\tt lenny@cmc.edu}

\medskip
S. R. Garcia, Department of Mathematics, Pomona College,

610 N. College Ave, Claremont, CA 91711, USA

{\tt stephan.garcia@pomona.edu}, URL: \url{http://pages.pomona.edu/~sg064747/}

\medskip
H. Maharaj, Department of Mathematics, Pomona College,

610 N. College Ave, Claremont, CA 91711, USA

{\tt hirenmaharaj@gmail.com}

\begin{thebibliography}{99}

\bibitem{bacher} R. Bacher,
Constructions of some perfect integral lattices with minimum 4.
{\em J. Th\'eor. Nombres Bordeaux} 27, 655--687 (2015).

\bibitem{magma} W. Bosma, J. Cannon, and C. Playoust,
The Magma algebra system. I. The user language.
{\em J. Symbolic Comput.}, 24, 235--265 (1997).

\bibitem{bo_et_al}
A.~B\"ottcher, L.~Fukshansky, S.~R.~Garcia, and H.~Maharaj,
On lattices generated by finite Abelian groups.
{\em SIAM J. Discrete Math.} 29,  382--404 (2015).

\bibitem{delsarte}
P. Delsarte, J. M. Goethals, and J. J. Seidel,
Spherical codes and designs.
{\em Geometriae Dedicata} 6 (3),  363--388 (1977).

\bibitem{lenny}
L. Fukshansky,
\newblock Integral orthogonal bases of small height for real polynomial spaces.
\newblock {\em Online J. Anal. Comb.} 4, 10 pp. (2009).

\bibitem{fu_ma}
L.~Fukshansky and H.~Maharaj,
Lattices from elliptic curves over finite fields.
{\em Finite Fields Appl.} 28, 67--78 (2014).

\bibitem{paulsen}
R. B. Holmes and V. I. Paulsen,
\newblock Optimal frames for erasures.
\newblock {\em Linear Algebra Appl.} 377, 31--51 (2004).

\bibitem{huppert}
B. Huppert,
\newblock {\em Character Theory of Finite Groups}.
\newblock Walter de Gruyter, Berlin, 1998.

\bibitem{martinet}
J.~Martinet,
\newblock {\em Perfect Lattices in Euclidean Spaces}.
\newblock Springer-Verlag, Berlin, 2003.

\bibitem{nebe}
G. Nebe,
\newblock {Boris Venkov's theory of lattices and spherical designs}.
\newblock {\em In: Diophantine methods, lattices, and arithmetic theory of quadratic forms,} Contemp. Math. 587, Amer. Math. Soc., Providence, RI, 1--19 (2013).

\bibitem{achill}
A.~Sch\"urmann,
\newblock {Perfect, strongly eutactic lattices are periodic extreme}.
\newblock {\em Adv. Math.} 225 (5), 2546--2564 (2010).

\bibitem{venkov}
B. B. Venkov,
\newblock {R\'eseaux et designs sph\'eriques}.
\newblock {\em In: R\'eseaux euclidiens, designs sph\'eriques et formes modulaires,}
Monogr. Enseign. Math. 37, Enseignement Math., Geneva, 10--86 (2001).


\end{thebibliography}
\end{document}